\documentclass[12pt, reqno]{amsart}
\usepackage[bookmarksnumbered, plainpages]{hyperref}
\usepackage{amsmath,amssymb,graphicx,accents}
\usepackage{enumerate,mdwlist}
\usepackage{enumitem}
\usepackage{esint}
\usepackage[symbol]{footmisc}
\usepackage{mathtools}
\usepackage{mathrsfs}

\usepackage{tikz}

\usepackage{thm-restate}

\usepackage[justification=centering]{caption}

\numberwithin{equation}{section}

\usepackage{amsthm}

\usepackage{verbatim}

\usepackage{nag}

\usepackage{tikz-cd} 
\usepackage{tikz-3dplot}

\usepackage{pgfplots}
\pgfplotsset{compat=1.18}

\makeatletter
\DeclareRobustCommand{\hvec}[1]{{\mathpalette\hvec@{#1}}}
\newcommand{\hvec@}[2]{%
  \vbox{\offinterlineskip
    \ialign{%
      \hfil##\hfil\cr
      $\m@th#1{}_{\rightharpoonup}$\kern-\scriptspace\cr
      $\m@th#1#2$\cr
    }%
  }%
}
\makeatother

\newcommand{\fdim}{\text{dim}_{\mathbf{F}}}

\DeclareMathOperator{\fordim}{\dim_{\mathbf{F}}}
\DeclareMathOperator{\adim}{\dim_{\mathbf{A}}}

\DeclareMathOperator{\RR}{\mathbf{R}}

\DeclareMathOperator{\NN}{\mathbf{N}}

\newtheorem{theorem}{Theorem}[section]
\newtheorem{lemma}[theorem]{Lemma}

\newtheorem{prop}[theorem]{Proposition}

\newtheorem{remark}[theorem]{Remark}

\allowdisplaybreaks

\begin{document}

\centerline{}

\title{The Maximum Codimension of a Salem Submanifold} % and More General Spectral Multipliers on Manifolds}
\author[Jacob Denson]{Jacob Denson$^*$}
\address{$^{*}$ University of Edinburgh, Edinburgh, Scotland, Jacob.Denson@ed.ac.uk}
%\dedicatory{This paper is dedicated to Professor ABCD}
%\date{Received: xxxxxx; Revised: yyyyyy; Accepted: zzzzzz.}
%\newline \indent $^{*}$ Corresponding author}

\begin{abstract}
    We determine a geometric condition necessary and sufficient for an $m$-dimensional manifold in Euclidean space to support a probability measure $\mu$ satisfying the Fourier decay bound $|\widehat{\mu}(\zeta)| \lesssim_\varepsilon |\zeta|^{\varepsilon - m/2}$ for all $\varepsilon > 0$. As a result, for each $m > 0$, we explicitly determine the largest codimension of an $m$-dimensional smooth submanifold $M$ of Euclidean space which is a Salem set. This largest codimension is precisely expressible in terms of the Radon-Hurwitz numbers. In particular, we find that the only odd dimensional manifolds which can be Salem sets are hypersurfaces, and that the largest codimension of an $m$-dimensional manifold which is a Salem set is upper bounded by $2 \lg(m/2) + 3$, and equal to $2 \lg(m/2) + 3$ when $m$ is a power of 16. The proof strategy, which involves covering manifolds by certain stationary sets associated with the Fourier transform on that manifold, is robust, and we demonstrate its use by proving that all nondegenerate curves in $\mathbf{R}^n$ have Fourier dimension equal to $2/n$, and find an alternate proof of a result of Junjie Zhu on the Fourier dimension of hypersurfaces with a fixed number of vanishing principal curvatures.

%    in Euclidean space supporting a probability measure $\mu$ satisfying the Fourier decay bound . In other words, we calculate the largest codimension of a $m$-dimensional manifold which is a Salem set. 
\end{abstract}

\maketitle

\vspace{-1em}

\section{Introduction}

The \emph{Fourier dimension} of a set $S \subset \RR^d$, denoted $\fdim(S)$, is the supremum of all $s \in (0,d]$ for which there exists a finite, non-negative Borel measure $\mu$ satisfying $|\widehat{\mu}(\zeta)| \lesssim |\zeta|^{-s/2}$ uniformly for all $\zeta \in \RR^d$. The Fourier dimension of a set is always smaller that its Hausdorff dimension, and a set is \emph{Salem} if it's Fourier dimension and Hausdorff dimension coincide. This paper investigates the Fourier dimension of submanifolds of Euclidean space.

Among sets of a given Hausdorff dimension, Salem sets exhibit subtle structural properties, reflecting their lack of correlation with plane waves. %In particular, Salem sets often lack `arithmetic structure', and have properties which might lead us to think of the set as being `curved' in some sense.
The Fourier dimension of hypersurfaces is often characterized by the curvature of the surface. For instance, hyperplanes have Fourier dimension zero, whereas hypersurfaces with non-vanishing Gauss curvature are Salem sets. Between these two extremes lie manifolds of higher codimension, where curvature is not the only property determining Fourier decay. The connection between geometric properties and Fourier decay then becomes quite delicate. %See \cite{Ekstrom} for a survey of results on the relations between Fourier dimension and geometry.

While the Fourier transform of smooth densities on manifolds can be analyzed using the well-developed body of techniques for the analysis of oscillatory integrals, such analysis does not preclude the existence of singular measures that may have faster Fourier decay than any smooth density. This possibility is difficult to eliminate, which makes obtaining upper bounds on the Fourier dimension of manifolds quite difficult.

Nonetheless, the existence of singular measures on a manifold with faster Fourier decay than smooth densities would have significant analytic consequences. In Fourier restriction theory, such decay properties would give rise to stronger Stein-Tomas type restriction estimates \cite{BakSeeger,Mitsis,Mockenhaupt}. In number theory, such decay would show that lattice points distribute non-homogeneously on the manifold \cite{Hlawka1,Hlawka2}. And in spectral theory, such measures would give rise fractal uncertainty principles on their support \cite{BourgainDyatlov,Dyatlov}. So identifying whether such measures could exist is of interest in analysis.

The main result of this paper is a simple geometric condition that characterizes when a smooth submanifold of Euclidean space is Salem. It implies that if a manifold is Salem, then smooth volume densities always have the fastest Fourier decay among all measures supported on the manifold.

% and in Section \ref{sec:lower_bounds_for_the_codimension_threshold} we will see why they emerge in the analysis of Fourier decay of measures supported on manifolds.

% the smallest integer such that any collection of $\rho(n)$ vector fields on $S^{n-1}$ must be linearly dependent at some point. J.F. Adams \cite{Adams}, in an early application of $K$-theory, showed that One may view Adam's result as a generalization of Poincar\'{e} and Brouwer's hairy ball theorem, since for odd $n$, $\rho(n) = 1$, reflecting the fact that any vector field on $S^{n-1}$ must vanish at some point. The larger growth $2^c$ for $0 \leq c < 3$ may be viewed as a manifestation of the existence of the division algebras $\CC$, $\mathbf{H}$, and $\mathbf{O}$ in two, four, and eight dimensions. 

\begin{theorem}\label{thm:blehtheorem}
    Let $M$ be a smooth, $m$-dimensional submanifold of Euclidean space. Then $M$ is a Salem set if and only if there exists a point $z_0 \in M$ so that for each $\zeta_0 \in (T_{z_0} M)^\perp$, the function $z \mapsto \zeta_0 \cdot (z - z_0)$ on $M$ has a non-degenerate critical point at $z_0$. In particular, if no point in $M$ has this property then it can have Fourier dimension no greater than $m - 1/3$.
\end{theorem}

\begin{remark} \label{remark:aiodjwaoidjawoidjawoiej12oi34j12io312}
    The upper bound in Theorem \ref{thm:blehtheorem}  is sharp, as one may verify that the graph of the function $\phi(x,t) = |x|^2 + t^3$ on $\RR^{m-1} \times \RR$ is an $m$-dimensional hypersurface of $\RR^{m+1}$ with Fourier dimension $m - 1/3$. See Proposition \ref{prop:ajwidhwaiudjaiouwdjwaqoij12oi321label} for details. 
\end{remark}

\begin{remark}
    In Section \ref{sec:required_smoothness_in_main_results}, we see that standard methods for analyzing oscillatory integrals justify that if the geometric condition in Theorem \ref{thm:blehtheorem} is satisfied, then appropriately localized smooth volume densities $\mu$ on the manifold $M$ satisfy $|\widehat{\mu}(\xi)| \lesssim |\xi|^{\varepsilon-m/2}$ for all $\varepsilon > 0$. The main difficulty in the proof of Theorem \ref{thm:blehtheorem} is instead justifying that manifolds which are Salem sets must satisfy the given geometric condition.
\end{remark}

%
% phi(x,y) = x^T A x + y^3
%  
%
% int e^{2 pi i ( [xi * x + zeta Q(x) ] + [eta y + zeta y^3]} a(x) a(y)
%   Integrate in x, get O( max(xi,zeta)^{-(m-1)/2} )
%   Integrate in y, get O( max(eta,zeta)^{-1/3} )
%   

As a result, for each $m$ we determine explicitly the largest integer $k$ for which there exists an $m$-dimensional smooth submanifold in $\RR^{m + k}$ that is a Salem set. It turns out that these quantities can be expressed precisely by the \emph{Radon-Hurwitz numbers}. For each $m \geq 1$, let $\rho(m)$ denote the $m$th Radon-Hurwitz number. If $2^{c + 4d}$ is the largest power of two dividing $m$, where $0 \leq c < 3$, then $\rho(m)$ is defined to be equal to $2^c + 8d$. Such numbers famously arise in the classification of division algebras and the problem of determining the study of the maximal number of linearly independent vector fields on $S^{m-1}$. It is interesting that these quantities also determine precisely the threshold for the codimension of submanifolds which are Salem sets. To simplify the statement of the main result of the paper, we will also define $\rho(a) = 0$ if $a \not \in \NN$, though this is non-standard.

\begin{theorem} \label{thm:mainTheorem}
    If $m$ and $k$ are positive integers, then there exists an $m$-dimensional smooth submanifold of $\RR^{m+k}$ which is a Salem set if and only if $k \leq \rho(m/2) + 1$. If $k > \rho(m/2 + 1)$, then any smooth $m$-dimensional submanifold of $\RR^{m+k}$ has Fourier dimension at most $m - 1/3$.
\end{theorem}

In Figure \ref{figureTable}, we provide a table explicitly computing for each $m$ the largest Euclidean space $\RR^d$ which contains an $m$-dimensional submanifold that is a Salem set. If $a$ is the largest power of $2$ dividing $m$, then one may check explicitly that
\begin{equation} \label{diowjdoiqwjoijo4i14123}
    \rho(m/2) + 1 \leq 2 \lg(a/2) + 3 \leq 2 \lg(m/2) + 3.
\end{equation}
The upper bound $2 \lg(m/2) + 3$ in \eqref{diowjdoiqwjoijo4i14123} is achieved precisely when $m$ is a power of $16$. When $m$ is any other power of two, $\rho(m/2) + 1 = 2 \lg(m/2) + 1$, which is still close to achieving the upper bound in \eqref{diowjdoiqwjoijo4i14123}. However, for a typical $m$ the maximal co-dimension is much smaller than this upper bound, unless $m$ is largely a power of $2$. In particular, when $m$ is odd, hypersurfaces are the only $m$-dimensional submanifolds which are Salem sets.

\begin{figure}
\def\arraystretch{1.3}
\begin{tabular}{ | c | c c c c c c c c | }
    \hline
    $m$ & $1$ & $2$ & $3$ & $4$ & $5$ & $6$ & $7$ & $8$\\ % & $9$ & $10$ & $11$ & $12$ & $13$ & $14$ & $15$ & $16$\\
    \hline
    $\RR^d$ & $\RR^2$ & $\RR^4$ & $\RR^4$ & $\RR^7$ & $\RR^6$ & $\RR^8$ & $\RR^8$ & $\RR^{13}$\\ % & $\RR^{10}$ & $\RR^{12}$ & $\RR^{12}$ & $\RR^{15}$ & $\RR^{14}$ & $\RR^{16}$ & $\RR^{16}$ & $\RR^{25}$\\
    \hline
\end{tabular}
\caption{For $m \in \{ 1, \dots, 8 \}$, this table lists the largest ambient space $\RR^d$ containing an $m$-dimensional submanifold which is a Salem set.} \label{figureTable}
\end{figure}

The proof of Theorem \ref{thm:mainTheorem} is geometric in nature. We show that a finite measure $\mu$ supported on a manifold can exhibit large Fourier decay only if it assigns small measure to certain tubes covering the manifold. By covering the manifold with such tubes, we deduce that $\mu = 0$. The argument is robust and may extend to yield sharp upper bounds for the Fourier dimensions of other classes of submanifolds.

We provide two results that provide evidence that sharp bounds can be obtained with this method. We begin by applying the method to obtain a simple proof of a theorem of Junjie Zhu \cite{ZhuRank}, who obtained sharp Fourier dimension bounds for \emph{hypersurfaces of constant nullity}, that is, hypersurfaces in Euclidean space with a fixed number of non-vanishing curvatures at each point. This built upon the techniques of earlier work in \cite{ZhuConic} for analyzing conic and cylindrical hypersurfaces, further inspired by work of Jonathan Fraser, Terence Harris, and Nicholas Kroon \cite{FraserHarrisKroon}, who showed the light cone in $\RR^{n+1}$ has Fourier dimension $n-1$.

\begin{restatable}{theorem}{thmzhutheorem} \label{thm:zhutheorem}
  Suppose $M$ is a smooth hypersurface in $\RR^{n+1}$, and exactly $k$ principal curvatures vanish at each point of $M$. Then $\fordim(M) = n - k$.
\end{restatable}

The methods of \cite{FraserHarrisKroon}, \cite{ZhuConic} and \cite{ZhuRank} rely on an `affine smoothing strategy'. The idea is to exploit the symmetries of the underlying manifold to reduce the study of the Fourier decay of singular measures to the Fourier decay of smooth measures, where one can use standard oscillatory integral techniques. In contrast, our argument directly exploits the geometry of hypersurfaces with a fixed number of vanishing principle curvatures, i.e. that they can be locally foliated by $n-k$ dimensional hyperplanes. This yields a substantially simpler proof.

Our methods also yield sharp bounds on the Fourier dimensions of nondegenerate curves. Recall that a curve parameterized by a function $\gamma: I \to \RR^n$ is non-degenerate if at each $t_0 \in I$, the vectors $\partial_t \gamma(t), \dots, \partial_t^n \gamma(t_0)$ are linearly independent. This class of curves possesses fewer structural properties that the class of surfaces covered in Theorem \ref{thm:zhutheorem}.

\begin{theorem} \label{momentcurvetheorem}
    If $C \subset \RR^n$ is a smooth, non-degenerate curve, $\fordim(C) = 2/n$.
\end{theorem}

Note this result aligns with the results of Theorem \ref{thm:mainTheorem}, which guarantees that all curves in $\RR^n$ for $n \geq 3$ must have Fourier dimension at most $2/3$.

Finally, we remark that the only smooth diffeomorphisms of $\RR^d$ that preserve the geometric property in Theorem \ref{thm:mainTheorem} for all submanifolds $M$ are affine transformations. As a corollary, we conclude that the only smooth diffeomorphisms on $\RR^d$ for $d > 1$ that preserve Fourier dimension are the affine transformations. This contrasts strongly with the study of subsets of $\RR^1$, where it remains an open question if \emph{all} smooth diffeomorphisms preserve Fourier dimension \cite{Ekstrom}.

% TODO Discuss results of Kahane. Consider $\{ X_a : a \in \RR \}$ is the fractional Brownian field on $\RR^d$, satisfying the variance identity $\VV(X_a - X_b) = |a - b|^\gamma$, for $0 \leq \gamma \leq 2$. Then the Brownian field is almost surely H\"{o}lder continuous of order $\gamma/2 - \varepsilon$ for each $\varepsilon > 0$. If $E$ has Hausdorff dimension at most $\gamma (d/2)$, then $E_X = \{ X_a: a \in E \}$ is almost surely a Salem set. If the Hausdorff dimension of $E$ is greater than $\gamma (d/2)$, then $E_X$. If $M$ is a manifold, then $E_M$ is a $C^{\gamma/2 - \varepsilon}$ manifold? In TODO, Fredrik Ekstr\"{o}m showed that for Borel subset $E$ of the real line, there exists a $C^1$ diffeomorphism 

\section{Required Smoothness in Main Results} \label{sec:required_smoothness_in_main_results}

The results proved in this paper become slightly more convoluted to state if we quantify the required smoothness of the manifolds we are studying, but this may still be of interest to some readers.

\vspace{0.5em}

\textbf{Theorem \ref{thm:blehtheorem}}: When $m \in \{ 1,2,3,4 \}$, the proof of Theorem \ref{thm:blehtheorem} desribed here carries through when $M$ is assumed to be a $C^{2 + \alpha}$-submanifold for some $\alpha > 0$. When $m > 4$, the proof of Theorem \ref{thm:blehtheorem} carries through only when $M$ is a $C^N$-submanifold for some integer $N \geq \lceil m/2 \rceil$. The additional regularity when $m > 4$ is needed to prove the required lower bounds on the Fourier dimension of $M$, i.e. to integrate by parts sufficiently many times to get Fourier decay away from stationary points in the oscillatory integrals considered in Section \ref{sec:lower_bounds_for_the_codimension_threshold}.

\vspace{0.5em}

\textbf{Theorem \ref{thm:mainTheorem}}: The proof of Theorem \ref{thm:mainTheorem} described here shows that if $k > \rho (m/2 + 1)$, and $\alpha > 0$, then any $C^{2 + \alpha}$-submanifold of $\RR^{m+k}$ has Fourier dimension at most $m - 1/3$.

\vspace{0.5em}

\textbf{Theorem \ref{thm:zhutheorem}}: The proof of Theorem \ref{thm:zhutheorem} described here carries through when $M$ is assumed to be a $C^2$ manifold of constant nullity  and $n - k \in \{ 1,2,3,4 \}$. For $n - k > 4$, the proof shows $C^2$ manifolds of constant nullity have Fourier dimension at least $n - k$, but to obtain the lower bound we must assume $M$ is $C^N$ for $N \geq \lceil (n - k)/2 \rceil$. Again, the trouble is obtaining enough regularity to perform integration by parts, justifying the required lower bounds.

\vspace{0.5em}

\textbf{Theorem \ref{momentcurvetheorem}}: the proof of Theorem \ref{momentcurvetheorem} given here carries through for all nondegenerate $C^n$-curves in $\RR^n$. It is difficult to change this level of regularity significantly because the usual definition of a nondegenerate curve requires a curve to have at least $n$ derivatives at each point.

\vspace{0.5em}

We remark that if one reformulates Theorems \ref{thm:blehtheorem}, \ref{thm:mainTheorem}, and \ref{thm:zhutheorem} to only discuss Fourier decay in directions normal to the manifolds in study, then one need not use the principle of non-stationary phase at all in the proofs of lower bounds (all oscillatory integrals considered will have stationary points), and then Theorems \ref{thm:blehtheorem} and \ref{thm:mainTheorem} follow for all $C^{2 + \alpha}$-submanifolds, and Theorem \ref{thm:zhutheorem} holds for all $C^2$-submanifolds.

\section{Lower Bounds For The Codimension Threshold} \label{sec:lower_bounds_for_the_codimension_threshold}

In this section, we prove the required \emph{lower} bounds required in Theorem \ref{thm:mainTheorem}. That is, for $m \geq 1$ and $k = \rho(m/2) + 1$ we prove the existence of an $m$-dimensional submanifold of $\RR^{m+k}$ that is a Salem set. The process of proof also provides significant motivation as to why these lower bounds are also the relevant upper bounds, and leads to a discussion as to \emph{why} the Radon-Hurwitz numbers appear in  Theorem \ref{thm:mainTheorem}.

Let $M$ be an arbitrary $m$-dimensional smooth submanifold of $\RR^d$. It is a heuristic in harmonic analysis that Fourier decay and smoothness are closely intertwined. This leads us to conjecture that the optimal decay of measures supported on $M$ may be achieved by a smooth volume density on $M$, i.e. by what might be called a \emph{Knapp type example} in restriction theory. Working locally, we may write $M$ as the graph of a function $\phi: \RR^m \to \RR^k$. If $\mu$ is supported on $M$, then we may find a smooth, compactly supported function $a$ on $\RR^m$ so that
\begin{equation} \label{aiodjwaoidjawioje211}
    \widehat{\mu}(\zeta) = \int a(x) e^{2 \pi i (\xi \cdot x + \eta \cdot \phi(x))}\; dx \quad\text{where $\zeta = (\xi,\eta)$}.
\end{equation}
Conversely, for any $a \in C_c^\infty(\RR^m)$ we can find a finite Borel measure $\mu$ on $M$ such that $\widehat{\mu}$ is expressed as \eqref{aiodjwaoidjawioje211}. To study optimal Fourier decay we are thus lead to the analysis of a standard oscillatory integral with smooth amplitude and phase, and so we may employ the principles of stationary and non-stationary phase.

Naively, stationary phase tells us that if the Hessian of the phase is non-singular at any critical point of the phase, then we can obtain an estimate $|\widehat{\mu}(\zeta)| \lesssim |\zeta|^{-m/2}$, which is sufficient to conclude that $M$ has Fourier dimension $m$, and is thus a Salem set. For any point $x_0$ in the support of $a$, and any $\eta$, setting $\xi = - D\phi(x_0)^t \eta$ in \eqref{aiodjwaoidjawioje211} gives rise to an oscillatory integral with a critical point at $x_0$. Define
\begin{equation} \label{eq:HessDefinition}
    H(\phi,\eta,x) = \sum\nolimits_{i = 1}^k \eta_i \text{Hess}(\phi_i)(x_0),
\end{equation}
where $\text{Hess}(f)(x_0)$ denotes the Hessian of a function $f$ at a point $x_0$. Then the Hessian of the phase at the critical point $x_0$ above is a multiple of $H(\phi,\eta,x_0)$. In order to assume all critical points are non-degenerate, we are thus forced to make the assumption that for any $x_0$ in the support of $a$, and all non-zero $\eta \in \RR^k$, $H(\phi,\eta,x_0)$ is non-singular. Conversely, under these assumptions we obtain the required Fourier decay.

\begin{lemma} \label{Lemma:OscillatoryIntegralNondegenerateLemma}
    Let $\phi: \RR^m \to \RR^k$ be a smooth function, and $a \in C_c^\infty(\RR^m)$, and define
    \[ I(\xi,\eta) = \int a(x) e^{2 \pi i (\xi \cdot x + \eta \cdot \phi(x))}\; dx. \]
    Suppose that $H(\phi,\eta,x)$ is non-singular for all $x \in \text{supp}(a)$ and all $\eta \neq 0$, where $H$ is defined in \eqref{eq:HessDefinition}. Then
    \[ |I(\zeta)| \lesssim_\varepsilon |\zeta|^{\varepsilon - m/2} \quad\text{for all $\varepsilon > 0$ and all $\zeta = (\xi,\eta)$ in $\RR^{m+k}$}. \]
\end{lemma}
\begin{proof}
    Let $\lambda = |\zeta|$. If $\lambda < 10$, then the trivial bound $|I(\xi,\eta)| \lesssim 1$ proves what is required, so we may assume in what follows that $\lambda \geq 10$, and also that $\varepsilon \leq 1$. By assumption, any critical point of the oscillatory integral is non-degenerate. Therefore, by breaking up $a$ into $O(1)$ functions supported on smaller sets, we may assume without loss of generality that, if a critical point for the oscillatory integral exists, then it is unique. We then consider two regimes:
\begin{description}[leftmargin=0.5cm,font=\mdseries\itshape,itemsep=10pt]
    \item[The Non-Stationary Regime] Suppose
    \begin{equation}
        | \xi + D \phi(x)^t \eta | \geq \lambda^{\varepsilon/2 - 1} \quad\text{for all $x \in \text{supp}(a)$}.
    \end{equation}
    Then successive integration by parts yields that
    %
    % f to be in C^{m/2 + 1}
    \begin{equation} \label{eq:finishingproofdecaypart1}
        |I(\zeta)| \lesssim_N \lambda^{-N} \quad\text{for all $N > 0$,}
    \end{equation}
    which is a sufficient bound if $N \geq m/2 - \varepsilon$.

    \item[The Stationary Regime] Suppose that
    \begin{equation} \label{eq:xiclosetocriticalpoint}
        | \xi + D\phi(x_0)^t \eta| \leq \lambda^{\varepsilon/2 - 1} \quad\text{for some $x_0 \in \text{supp}(a)$.}
    \end{equation}
    Let $\xi_0 = - D\phi(x_0)^t \eta$. If we define $a_0(x) = a(x) e^{2 \pi i \lambda (\xi - \xi_0) \cdot x}$, then \eqref{eq:xiclosetocriticalpoint} implies
    \begin{equation} \label{eq:alpha0derivativebounds}
        |\partial^\beta a_0| \lesssim \lambda^{|\beta| \varepsilon / 2} \quad\text{for all multi-indices $\beta$.}
    \end{equation}
    We should therefore view $a_0$ as smooth, up to a $\varepsilon$-loss in $\lambda$. We can then write
    \begin{equation}
        I(\zeta) = \int a_0(x) e^{2 \pi i \lambda( \xi_0 \cdot x + \eta \cdot \phi(x) )}\; dx.
    \end{equation}
    The phase of this oscillatory integral has a unique non-degenerate critical point at $x_0$. The principle of stationary phase (see, for instance, Theorem 7.7.5 of \cite{HormanderVol1}), together with \eqref{eq:alpha0derivativebounds} thus implies
    \begin{equation} \label{eq:iofejioruvnweicne2iodnio12312}
        |I(\zeta)| \lesssim \Big[ \sup\nolimits_{|\beta| \leq 2} \| \partial^\beta a_0 \|_{L^\infty(\RR^d)} \Big] |\eta|^{-m/2} \lesssim \lambda^\varepsilon |\eta|^{-m/2}.
    \end{equation}
    Since \eqref{eq:xiclosetocriticalpoint} implies
    \begin{equation}
        \lambda \leq |\xi| + |\eta| \leq \lambda^{\varepsilon/2 - 1} + |D\phi(x_0)^t \eta| + |\eta| \leq \lambda^{\varepsilon/2 - 1} + O(|\eta|)
    \end{equation}
    we conclude since $\lambda \geq 10$ and $\varepsilon \leq 1$ that $\lambda \lesssim |\eta|$, and thus from \eqref{eq:iofejioruvnweicne2iodnio12312} that
    \begin{equation} \label{eq:finishingproofofdecaypart2}
        |\widehat{\mu}(\zeta)| \lesssim \lambda^{\varepsilon - m/2},
    \end{equation}
    which is sufficient for our purposes.
\end{description}
Combining \eqref{eq:finishingproofdecaypart1} and \eqref{eq:finishingproofofdecaypart2} yields the claim.
\end{proof}

We now apply this result to smooth volume densities supported on manifolds, and obtain the following claim.

\begin{prop} \label{prop:SalemManifoldIfNondgenerate}
    Suppose $M \subset \RR^{m+k}$ is the graph of a smooth function $\phi: \RR^m \to \RR^k$. Suppose that there exists a point $x_0 \in \RR^m$ so that $H(\phi,\eta,x_0)$ is non-singular for all $\eta \neq 0$, where $H$ is defined in \eqref{eq:HessDefinition}. Then $M$ is a Salem set.
\end{prop}
\begin{proof}
    The assumption, by compactness, continuity, and homogeneity of $H$ in $\eta$, implies that there exists a neighborhood $U$ of $x_0$ so that $H(\phi,\eta,x)$ is non-singular for all $x \in U$ and $\eta \neq 0$. If $a \in C_c^\infty(\RR^m)$ is supported on $U$, and $I$ is defined as in Lemma \ref{Lemma:OscillatoryIntegralNondegenerateLemma}, then that lemma implies $|I(\zeta)| \lesssim_\varepsilon |\zeta|^{\varepsilon - m/2}$  for all $\varepsilon > 0$. But there exists a finite-nonnegative measure $\mu$ supported on $M$ such that $\widehat{\mu}(\zeta) = I(\zeta)$ for all $\zeta \in \RR^d$, and so we have proved that $M$ has Fourier dimension at least $m$. The Hausdorff dimension of $M$ is $m$, and always upper bounds the Fourier dimension, and so it immediately follows that $M$ is a Salem set.
\end{proof}

To determine how large the codimension of a manifold must be before it fails to be Salem, we are thus lead to the following question: for a given $m$, what is the largest $k$ such that there exists a function $\phi: \RR^m \to \RR^k$ and $x \in \RR^m$ so that $H(\phi,\eta,x)$ is non-singular for all $x \neq 0$. Given \emph{any} symmetric $m \times m$ matrices $A_1,\dots,A_k$, if we define
\begin{equation}
    \phi(x) = \left( x^t A_1 x / 2, \dots, x^t A_k x / 2 \right),
\end{equation}
then $H(\phi,\eta,x) = \eta_1 A_1 + \cdots + \eta_k A_k$. We thus see our problem is equivalent to a completely algebraic problem: for a given $m$, what is the largest $k$ so that there exists $k$ symmetric $m \times m$ matrices $A_1,\dots,A_k$, so that for any $\eta \neq 0$, the matrix $\eta_1 A_1 + \cdots + \eta_k A_k$ is non-singular. Equivalently, what is the largest dimension of a subspace of the space of all $m \times m$ symmetric matrices so that all non-zero elements of this subspace are non-singular. Fortunately, this problem has been fully resolved by methods of representation theory and algebraic topology, and we briefly indulge in an outline of the history of this problem.

The original incarnation of the problem, connected with the theory of quadratic forms in number theory, was to find the largest $k$ such that there exists $k$ orthogonal $m \times m$ matrices $O_1,\dots,O_k$ so that $\eta_1 O_1 + \cdots + \eta_k O_k$ is orthogonal whenever $|\eta| = 1$. In 1898, Adolf Hurwitz \cite{Hurwitz} proved that $k$ was equal to $m$ only when $m \in \{ 1, 2, 4, 8 \}$, and two decades later Johann Radon \cite{Radon} extended this method to prove that $k$ was always at least $\rho(m)$ for all $m$, where $\rho(m)$ are the Radon Hurwitz numbers defined in the introduction. %and equal to $\rho(m)$ when $m$ was a power of two.

To obtain the required upper bounds, algebraic topology enters the picture. To get a feeling why the problem has a topological nature, let us consider a short proof as to why $k = 1$ when $m$ is odd. Consider any two non-singular matrices $A_1$ and $A_2$. We will show some non-trivial linear combination of the matrices must be singular. Consider the continuous matrix-valued function
\begin{equation}
    A(\eta) = \eta_1 A_2 + \eta_2 A_2
\end{equation}
defined on $S^1$. If $f(\theta)$ is the determinant of $A(\cos(\theta), \sin(\theta))$, then
\begin{equation}
    f(0) = \det(A_1) \quad\text{and}\quad f(\pi) = - \det(A_1).
\end{equation}
The intermediate value theorem implies there exists some $\theta_0 \in [0,\pi]$ so that $f(\theta_0) = 0$, and so we conclude $\cos(\theta_0) A_1 + \sin(\theta_0) A_2$ is singular, as required.

Of course, the analysis of the more general result, where $k$ can be arbitrarily large, requires a much more robust set of topological machinery In the 1940s Beno Eckmann \cite{Eckmann} and Heinz Hopf \cite{Hopf} independently made the following observation. Suppose $A_1,\dots,A_k$ are matrices such that $A(\eta) = \sum \eta_i A_i$ is non-singular when $\eta \neq 0$. By multiplying each $A_i$ on the left by $A_k$, we may assume without loss of generality that $A_k = I$. It then follows that the vector fields
\begin{equation}
    X_1,\dots,X_{k-1}: S^{m-1} \to \RR^d \quad\text{defined by}\quad X_i(x) = \Pi_x (A_i x)
\end{equation}
are everywhere linearly independent, where $\Pi_x$ is orthogonal projection onto $T_x S^{m-1} \subset \RR^m$. Indeed the vectors $A_1 x, \dots, A_k x$ must be linearly independent, and since $A_k x = x$ it follows that no linear combination of $\{ A_i x : i < k \}$ can be a multiple of $x$. So we obtain a family of $k-1$ vector fields on $S^{m-1}$ that are everywhere linearly independent. We thus see that to control the maximum $k$ for which $A_1,\dots, A_k$ exist, we are forced to understand how vector fields on $S^{m-1}$ can remain linearly independent. This problem can be viewed in terms of \emph{$K$-theory}, the topological study of vector bundles defined on $S^{m-1}$.

Following work of Norman Steenrod, John Whitehead, and Hirosi Toda \cite{SteenroadWhitehead,Toda}, in 1962 John Adams \cite{Adams} was able to show that for $k > \rho(m)$ there does not exist a family of $k-1$ everywhere linearly independent continuous vector fields on $S^{m-1}$. This completed the proof that $\rho(m)$ was exactly the largest dimension of a subspace of matrices any non-zero element of which was non-singular for any $m$. Via a simple reduction to the non-singular case, in 1965 John Adams, Peter Lax, and Ralph Phillips \cite{AdamsLaxPhillips} then showed that $\rho(m/2) + 1$ is the largest dimension of the subspace of $m \times m$ symmetric matrices such that any non-zero element is non-singular.

Using Proposition \ref{prop:SalemManifoldIfNondgenerate}, the following results thus immediately follow.

%Adams' landmark result \cite{Adams} explicitly determined the values of $\rho(m)$ for all $m \geq 1$, and shoed that $\rho(m)$ was also the largest dimension of a subspace of $m \times m$ matrices such that any non-zero element is invertible. Later, Adams, Lax, and Phillips \cite{AdamsLaxPhillips} showed that $\rho(m/2) + 1$ 

\begin{prop} \label{prop:ExistenceofSalemCodimension}
    If $k \leq \rho(m/2) + 1$, then there exists a smooth $m$-dimensional submanifold of $\RR^{m+k}$ which is a Salem set.
\end{prop}

\begin{lemma} \label{lemma:StationaryDegenerateHighCodimensionLemma}
    If $k > \rho(m/2) + 1$, and $M$ is any $m$-dimensional $C^2$-submanifold of $\RR^{m+k}$, then for each $z_0 \in M$, there exists $\zeta \in (T_{z_0} M)^\perp$ and $v \in T_{z_0} M$ so that the function $\zeta \cdot (z - z_0)$ vanishes to order three at $z_0$ in the direction $v$.
%     $\phi: \RR^m \to \RR^k$ is any $C^2$ function, then for each $x \in \RR^m$ there exists $\eta \in \RR^k$ so that $H(\phi,\eta,x)$ is singular, where $H$ is defined in \eqref{eq:HessDefinition}.
\end{lemma}

We have seen how the study of smooth volume densities on a manifold leads to a conjecture on the largest codimension of a Salem $m$-dimensional manifold. Using Lemma \ref{lemma:StationaryDegenerateHighCodimensionLemma}, it is not too difficult using standard oscillatory integral techniques %\footnote{For instance, the Esse\'{e}n concentration inequality implies this (see Lemma 7.17 of \cite{TaoVu}).}
to show that no smooth volume density $\mu$ on an $m$-dimensional $C^2$-manifold $M$ of codimension $k > \rho(m/2) + 1$ can satisfy a Fourier decay bound of the form $|\widehat{\mu}(\zeta)| \lesssim |\zeta|^{-s/2}$ for $s > m - 1/3$. However, the difficulty in the remainder of the argument is showing that one cannot get this decay for \emph{any} measure, no matter how non-smooth. The remainder of the argument consists of showing directly that no finite Borel measure supported can have faster Fourier decay then one obtained by a smooth volume density.

To prove this, we will take an arbitrary Borel measure $\mu$ on $M$, and show that the only way to have an estimate of the form $|\widehat{\mu}(\zeta)| \lesssim |\zeta|^{-s/2}$ for $s$ sufficiently close to $m$ is if $\mu = 0$. We will do this by showing $\mu$ must assign small measure to a family of elongated tubes on the set $M$, which efficiently cover $M$, and the only way this is possible if $s > m - 1/3$ is if $\mu = 0$. The upper bound $\fordim(M) \leq m - 1/3$ immediately follows.

To conclude this section, we provide two additional propositions relevant to Remark \ref{remark:aiodjwaoidjawoidjawoiej12oi34j12io312} and Theorem \ref{momentcurvetheorem}.

\begin{prop} \label{prop:ajwidhwaiudjaiouwdjwaqoij12oi321label}
    The graph of the function $\phi(x,t) = |x|^2 + t^3$ on $\RR^{m-1} \times \RR$ has Fourier dimension at least $m-1/3$.
\end{prop}
\begin{proof}
    As above, we reduce to the analysis of oscillatory integrals. We may find a measure $\mu$ supported on the graph so that
    \begin{equation}
        \widehat{\mu}(\xi,\omega,\eta) = \int a_1(x) a_2(t) e^{2 \pi i ( \xi \cdot x + \omega t + \eta \cdot \phi(x,t) )}\; dx\; dt,
    \end{equation}
    where $a_1$ and $a_2$ are smooth and compactly supported. This integral tensorizes, i.e. we can write $\widehat{\mu}(\xi,\omega,\eta) = I_1(\xi,\eta) I_2(\omega,\eta)$, where
    \begin{equation}
        I_1(\xi,\eta) = \int a_1(x) e^{2 \pi i (\xi \cdot x + \eta |x|^2)}\; dx\ \text{and}\ I_2(\omega,\eta) = \int a_2(t) e^{2 \pi i (\omega t + \eta t^3)}\; dt.
    \end{equation}
    If $|\xi| \gtrsim |\eta|$ and $|\xi| \geq |\omega|$, the phase of the oscillatory integral $I_1$ is non-stationary, and we can integrate by parts to conclude that $|I_1(\xi,\eta)| \lesssim_N |\xi|^{-N}$ for arbitrarily large $N > 0$. Combining this with the trivial estimate $|I_2(\omega,\eta)| \lesssim 1$ we conclude that $|\widehat{\mu}(\zeta)| \lesssim_N |\zeta|^{-N}$. A similar analysis works if $|\omega| \gtrsim |\eta|$ and $|\omega| \geq |\xi|$, so we may assume that $|\eta| \gtrsim |\xi|$ and $|\eta| \gtrsim |\omega|$. The stationary phase principle then justifies that $|I_1(\xi,\eta)| \lesssim |\eta|^{-(m-1)/2}$, and the Van der Corput Lemma (see Chapter VIII, Proposition 2 of \cite{Stein}) justifies that $|I_2(\omega,\eta)| \lesssim |\eta|^{-1/3}$. Combining these estimates yields that $|\widehat{\mu}(\zeta)| \lesssim |\zeta|^{-m/2 + 1/6}$. Since $-m/2 + 1/6 = -s/2$ with $s = m-1/3$ this completes the proof.
\end{proof}

\begin{prop} \label{prop:dkpoawkdpoawdkpaowkkvmdsovselabel}
    If $C \subset \RR^n$ is a smooth, nondegenerate curve, $\fordim(C) \geq 2/n$.
\end{prop}
\begin{proof}
    Let $\gamma: I \to \RR^n$ be a smooth parameterization of some portion of the curve, such that the vectors $e_1(t_0) = \partial_t \gamma(t_0), \dots, e_n(t_0) = \partial_t^n \gamma(t_0)$ are linearly independent. We may find a measure supported on $C$ so that
    \[ \widehat{\mu}(\xi) = \int a(t) e^{2 \pi i \xi \cdot \gamma(t)}\; dt \quad\text{for all $\xi \in \RR^n$}, \]
    where $a$ is smooth and compactly supported in $I$. We claim that $|\widehat{\mu}(\xi)| \lesssim |\xi|^{-1/n}$, which would suffice to complete the proof. Since the vectors $\{ e_i(t) \}$ are spanning for each $t \in I$, we may find $\delta > 0$ so that for all $t \in I$ and all $\xi \in \RR^d$, there exists $i \in \{ 1, \dots, n \}$ so that $|\xi \cdot e_i(t)| \geq \delta |\xi|$. By breaking up $a$ into $O(1)$ functions supported on smaller regions, we may assume without loss of generality that $|e_i(t) - e_i(t')| \leq \delta/10$ for all $t$ and $t'$ in the support of $a$. It then follows that for any $\xi \in \RR^d$, we may find some $i \in \{ 1, \dots, n \}$ so that $|\xi \cdot e_i(t)| \geq \delta |\xi|$ for all $t$ in the support of $a$. If $i = 1$, then integration by parts then justifies that $|\widehat{\mu}(\xi)| \lesssim_N |\xi|^{-N}$ for all $N > 0$. For $i \in \{ 2, \dots, n \}$, the Van der Corput lemma justifies that $|\widehat{\mu}(\xi)| \lesssim |\xi|^{-1/i} \lesssim |\xi|^{-1/n}$. In either case, we conclude that the measure $\mu$ has the required Fourier decay.
\end{proof}

% ASK JIM TODAY: HAVE I PROVED SOMETHING BETTER THAN THE ESSEEN CONCENTRATION INEQUALITY?
% Esseen
%
%   R^d epsilon^{d-s/2} for epsilon > 1
%   R^d epsilon^d for epsilon < 1
%   
%
%   epsilon^{-d}
%
%   L_{R phi} << int_{-1}^1 | I(lambda R phi) | dlambda
%             << R^{-1} int_{-R}^R | I(lambda phi) | dlambda
%             << R^{-1} ( 1 + R^{1-s/2} )
%             << R^{-1} + R^{-s/2}      (SO IF s > 2 then the result is stronger)
%
%   For epsilon > r
%   r^d (epsilon / r)^{d-s/2} = r^{-s/2} epsilon^{d-s/2}
%   For epsilon < r
%           epsilon^d

% For instance, results of Mockenhaupt (TODO), Mitsis (TODO), Bak and Seeger (TODO) show that such decay estimates imply bounds for the extension operator
%
%\[ E_\mu f(x) = \int_S f(\xi) e^{2 \pi i \xi \cdot x}\; d\mu(\xi) \]
%
%of Stein-Tomas type, i.e. such that
%
%\[ \| E_\mu f \|_{L^p(\RR^d)} \lesssim \| f \|_{L^2(S,\mu)} \quad\text{for $p > 4d/s - 2$}. \]
%
%For a general set $S$, and a general measure $\mu$ which satisfies the Frostman condition $\mu(B) \lesssim \text{rad}(B)^s$ for all balls $B \subset \RR^d$, the Fourier transform of $\mu$ may not lie in $L^p(\RR^d)$ for any $p < \infty$, so that the Stein-Tomas estimate (TODO) only holds in the trivial case where $p = \infty$.

\section{Bounds on the Measure of Stationary Slabs}

In this section, we study an arbitrary Borel measure $\mu$ supported on a set $S$. Under the assumption that $\mu$ has Fourier decay, we take a slab $\Theta$ supported on a small neighborhood of the set $S$ pointing in some direction $\zeta_0$. Under the assumption that $S$ is suitably transverse to the direction $\zeta_0$, we show that $\mu(\Theta)$ must be small. When using these results in the sequel, we will take $S$ to be a manifold, and $\zeta_0$ a vector normal to the manifold at some point (from which the transversality assumption then follows). It is possible, however, the Lemma \ref{Lemmaoaijdwoiwajdioaw} may still have some application when $S$ is a more general set for which transversality still makes sense, such as a rectifiable surface.

\noindent \begin{minipage}{\textwidth}

\begin{lemma} \label{Lemmaoaijdwoiwajdioaw}
    Let $\mu$ be a finite Borel measure supported on a set $S \subset \RR^d$ such that $|\widehat{\mu}(\xi)| \leq C \langle \xi \rangle^{-s/2}$ for all $\xi \in \RR^d$. Fix $\varepsilon > 0$, $R > 0$, and $z_0 \in S$, let $V$ be a subspace of $\RR^d$, and for $z \in \RR^d$, let $\Pi$ and $\Pi^\perp$ be the orthogonal projections onto $V$ and $V^\perp$ respectively. Consider a unit vector $\zeta_0 \in V^\perp$, and suppose $Q \subset V$ is a symmetric ellipsoid so that
    \[ |\zeta_0 \cdot (z - z_0)| \leq 1/R \quad\text{when}\quad \Pi(z - z_0) \in Q^*\ \text{and}\ |\Pi^\perp(z - z_0)| \leq R^\varepsilon / R, \]
    where $Q^*$ is a dilation of $Q$ by a factor of $2$. Then if we consider the slab
    \[ \Theta = \{ z \in \RR^d: \Pi(z - z_0) \in Q\ \text{and}\ |\Pi^\perp(z - z_0)| \leq 1/R \}, \]
    then $\mu(\Theta) \lesssim C R^{-s/2}$, with an implicit constant depending only on $d$ and $\varepsilon$.
\end{lemma}

\end{minipage}

\begin{figure}
\begin{tikzpicture}[scale=0.8, line cap=round]

% --- Rectangle ---
\draw[thick] (-1, -0.5) rectangle (1, 0.5);

% --- Central point ---
\fill (0,0) circle (2pt);
\node[below right] at (0,0) {$z_0$};

% --- Single arrow from center ---
\draw[->, thick] (0,0) -- (0.0,0.8);
\node[right] at (0,0.8) {$\zeta_0$};

% --- Dashed construction boundary ---
\draw[dashed] (-1,-2) rectangle (1,2);

% --- Brownian-like top curve ---
\draw[line width=0.5mm]
  plot[smooth] coordinates {
(-2.2,0.43)
(-2.1,0.4)
(-2,0.4323406237806613)
(-1.9,-0.4794400339778237*1.8)
(-1.8,-0.34966984574321963*1.8)
(-1.7,-0.8702168521024216*1.8)
(-1.6,-0.5891287743898475*1.8)
(-1.5,-0.48736699306855574*1.8)
(-1.4,-1.1392685847942174*1.8)
(-1.3,-0.9947769392227928*1.8)
(-1.2,-0.6078871301893458*1.8)
(-1.1,-0.3691789001923906*1.5)  
  (-1,-0.3076490834936588*1)
(-0.9,-0.2411719263694608*1.2)
(-0.8,-0.19786663290363526*1.2)
(-0.7,-0.12675601858002556*1.2)
(-0.6,-0.1376824569451548*1.2)
(-0.5,-0.05452776976874885*1.2)
(-0.4,0.008580611507976688*1.2)
(-0.3,0.0048367547963774404*1.2)
(-0.2,-0.03043055057932733*1.2)
(-0.1,0)
    (0,0)
    (0.1,0)    
    (0.2,0.07978170592162184*3)
    (0.3,-0.003079693525864391*3)
    (0.4,0.0868786281021823*3)
    (0.5,0.05452385684315648*3)
    (0.6,0.04003071231504238*3)
    (0.7,0.006437820285207341*3)
    (0.8,0.05807728562112096*3)
    (0.9,0.06714776757917238*3)
    (1,0.13570587081068581*3)
    (1.1,0.4071176124320574)
    (1.2,0.09038716042809697*2)
    (1.3,-0.7692061830122225*2)
    (1.4,0.08539056524848632*2)
    (1.5,-0.09185167273312289*2)
    (1.6,-0.3374211748349123*2)
    (1.7,-0.45615403521153786*2)
    (1.8,0.23171162178601756*2)
    (1.9,0.38974376423196877*2)
    (2,0.3586106242763749*2)
  };

\end{tikzpicture}
\caption{An illustration of the setup to Lemma \ref{Lemmaoaijdwoiwajdioaw}. Here $V$ is the $x$-axis, $V^\perp$ the $y$-axis, and $S$ is the curve through $z_0$. The vector $\zeta_0$ is `transverse' to the set $S$ at $z_0$, in the sense that within the dashed box with height $R^\varepsilon R^{-1}$, the set $S$ only contains points in the shorter solid box with height $R^{-1}$.}
\end{figure}
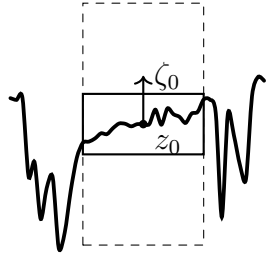

\begin{remark}
    Lemma \ref{Lemmaoaijdwoiwajdioaw} is a refinement of the Ess\'{e}en Concentration Inequality\footnote{See e.g. Lemma 7.17 of \cite{TaoVu} for the full Statement and proof of the Ess\'{e}en Concentration Inequality, the method first appearing in a 1966 paper of Carl-Gustav Ess\'{e}en \cite{Esseen}. Taking $\varepsilon = 1/R$ in the version of the inequality stated in Lemma 7.17 of \cite{TaoVu}, and rescaling the inequality non-isotropically so it applies to ellipsoids rather than just balls gives \eqref{eq:laawdopjawoidjawoij12oi3j12oi3bel}.}, and thus results to a general scheme in harmonic analysis which relates bounds for oscillatory integrals to control on the size of the \emph{stationary sets} associated with such oscillatory integrals. For any symmetric ellipsoid $Q \subset V$, if $\Theta$ is defined as in Lemma \ref{Lemmaoaijdwoiwajdioaw}, then applying the Ess\'{e}en Concentration Inequality to $\Theta$ shows that
    \begin{equation} \label{eq:laawdopjawoidjawoij12oi3j12oi3bel}
        \mu(\Theta) \lesssim \fint_{Q^* \times B_R^{V^\perp}} |\widehat{\mu}(\xi)|,
    \end{equation}
    where $Q^*$ is the dual ellipsoid to $Q$ centered at the origin, and $B_R^{V^\perp}$ is the ball centered at the origin in $V^\perp$ of radius $R$. Substituting the bound $|\widehat{\mu}(\xi)| \lesssim 1 + |\xi|^{-s/2}$ into \eqref{eq:laawdopjawoidjawoij12oi3j12oi3bel} and integrating, we find that
    \begin{equation}
        \mu(\Theta) \lesssim R^{- \dim(V^\perp)} \left( 1 + R^{\dim(V^\perp) - s/2} \right).
    \end{equation}
    The Ess\'{e}en Concentration Inequality thus implies Lemma \ref{Lemmaoaijdwoiwajdioaw} when $s \leq 2\dim(V^\perp)$. The transversality assumptions in Lemma \ref{Lemmaoaijdwoiwajdioaw} allow us to obtain the required bounds in the case $s > 2 \dim(V^\perp)$ as well, which is necessary in what follows.
\end{remark}

\begin{proof}[Proof of Lemma \ref{Lemmaoaijdwoiwajdioaw}]
    If $R \leq 2^{1/\varepsilon}$, then the bound
    \begin{equation}
        \mu(\Theta) \leq \mu(\RR^d) \leq C \lesssim C R^{-s/2}
    \end{equation}
    follows trivially, so we may assume that $R \geq 2^{1/\varepsilon}$ in what follows. Let $A: V \to V$ be a linear isomorphism mapping the unit ball to $Q$. Consider a symmetric Schwartz function
    \begin{equation}
        \chi(z) = (\chi_1 \circ A^{-1} \circ \Pi)(z - z_0) (\chi_2 \circ \Pi^\perp)(R^{1 - \varepsilon} (z - z_0)),
    \end{equation}
    where
    \begin{equation}
        \chi_1(x) = 0\ \text{for $|x| \geq 2$} \quad\text{and}\quad \chi_2(y) = 0\ \text{for $|y| \geq 1$},
    \end{equation}
    and such that
    \begin{equation} \label{awdiaoqwj12312}
        |\chi_1(x)|, |\chi_2(y)| \geq 1 \quad\text{for $|x| \leq 1$ and $|y| \leq 1/2$}.
    \end{equation}
    Define
    \begin{equation}
        \eta(z) = e^{i R \zeta_0 \cdot (z - z_0)} \chi(-z).
    \end{equation}
    We will bound $\mu(\Theta)$ by studying the convolution $\eta * \mu$.

    We calculate directly that
    \begin{equation} \label{aiodjwaoij123123}
        (\eta * \mu)(0) = \int \eta(-z) d\mu(z) = \int e^{- i R \zeta_0 \cdot (z - z_0)} \chi(z) d\mu(z).
    \end{equation}
    On the support of the integrand, $z \in S$, $\Pi(z - z_0) \in Q$ and $|\Pi^\perp(z - z_0)| \leq R^\varepsilon / R$. By assumption, this means that $|\zeta_0 \cdot (z - z_0)| \leq 1/R$, and thus that
    \begin{equation}
        \text{Re} \left(e^{-i R \zeta_0 \cdot (z - z_0)} \right) = \cos \left( R \zeta_0 \cdot (z - z_0) \right) \geq \cos(1) \geq 1/2.
    \end{equation}
    Thus the real part of the integrand in \eqref{aiodjwaoij123123} is non-negative. Moreover, we conclude from \eqref{awdiaoqwj12312} and that $R \geq 2^{1/\varepsilon}$ that
    % R >= 2^{1/eps}
    \begin{equation}
        \chi(z) \geq 1 \quad\text{if $z \in \Theta$},
    \end{equation}
    and so
    \begin{equation} \label{awidjawioeq2312}
    \begin{split}
    \text{Re}(\eta * \mu)(0) &\geq \int_{\Theta} \text{Re} \left( e^{-i R \zeta_0 \cdot (z - z_0)} \right) \chi(z) d\mu(z) \geq \mu(\Theta) / 2.
    \end{split}
    \end{equation}
    Since $\eta * \mu$ is a smooth function, rearranging \eqref{awidjawioeq2312} gives
    \begin{equation} \label{aiodjaoi1241421}
        \mu(\Theta) \leq 2 \| \eta * \mu \|_{L^\infty(\RR^d)}.
    \end{equation}
    In the remainder of the argument, we use the Fourier transform to justify that the upper bound $\| \eta * \mu \|_{L^\infty(\RR^d)} \lesssim R^{-s/2}$ holds, which will complete the proof.

    Define
    \begin{equation}
        \Omega = \{ \zeta: |\Pi^\perp(\zeta)| \leq R/10 \},
    \end{equation}
    Symmetries of the Fourier transform, and the smoothness of $\chi$, imply that
    \begin{equation} \label{djawoidjqioj21oi3j12io3j21}
        \| \widehat{\eta} \|_{L^1(\Omega)} \lesssim_N R^{-N}\ \text{for all $N \geq 0$,} \quad\text{and}\quad \| \widehat{\eta} \|_{L^1(\Omega^c)} \lesssim 1,
    \end{equation}
    where the implicit constants in \eqref{djawoidjqioj21oi3j12io3j21} depend only on $d$ and $\varepsilon$. Combining these two estimates with the bounds
    \begin{equation}
        \| \widehat{\mu} \|_{L^\infty(\Omega)} \lesssim C \quad\text{and}\quad \| \widehat{\mu} \|_{L^\infty(\Omega^c)} \lesssim CR^{-s/2},
    \end{equation}
    which follow from the Fourier decay of $\mu$, H\"{o}lder's inequality implies that
    \begin{equation} \label{eqawdkaiodjawqoidjqwodiqjwdioq}
        \| \widehat{\eta * \mu} \|_{L^1(\RR^d)} = \| \widehat{\eta}\; \widehat{\mu} \|_{L^1(\RR^d)} \lesssim R^{-s/2}.
    \end{equation}
The Hausdorff-Young inequality, \eqref{aiodjaoi1241421} and \eqref{eqawdkaiodjawqoidjqwodiqjwdioq} imply $\mu(\Theta) \lesssim C R^{-s/2}$.
\end{proof}

Our stategy to upper bounding the Fourier dimension of a manifold $M$ involves taking a measure $\mu$ with $|\widehat{\mu}(\zeta)| \lesssim |\zeta|^{-s/2}$ on $M$, and to cover $M$ by a family of $N(R)$ slabs $\{ \Theta_i \}$ to which we can apply Lemma \ref{Lemmaoaijdwoiwajdioaw} for appropriate $\zeta_i$ and $R_i$. Then we conclude from Lemma \ref{Lemmaoaijdwoiwajdioaw} that
\begin{equation} \label{awoidjawoidjawoidjawiodj123123}
    \mu(M) \leq \sum\nolimits_i \mu(\Theta_i) \lesssim \sum\nolimits_i R_i^{-s/2}.
\end{equation}
If for each $\varepsilon > 0$, we can choose $\{ \Theta_i \}$ with $\sum_i R_i^{-s/2} \leq \varepsilon$, then taking $\varepsilon \to 0$ yields that $\mu(M) = 0$. We thus conclude that $\fordim(M) \leq s$.

One might compare this process to another notion of dimension, known as \emph{affine dimension}, introduced by Daniel Oberlin \cite{Oberlin}, which is sensitive to the same kind of geometric features of a manifold as the Fourier dimension. The affine dimension is defined very similarly to the Hausdorff dimension -- for a particular $s > 0$, we define an outer measure, $A^s(E)$ for a set $E \subset \RR^d$, by setting
\begin{equation}
    A^s(E) = \lim_{\delta \to 0} \inf \left\{ \sum\nolimits_i |Q_i|^{s/d} \right\},
\end{equation}
where the infimum is taken over all families of ellipsoids $\{ Q_i \}$ which cover $E$, with each ellipsoid in the cover having diameter at most $\delta$. The main distinction here between Hausdorff measure is that the sets $\{ Q_i \}$ are allowed to be ellipsoids rather than just balls. The affine dimension of a set $E$, which we denote by $\adim(E)$, is then the supremum of all $s > 0$ so that $A^s(E) = 0$. The quantity is natural to consider when studying bounds on convolution and restriction operators, but has slightly unusual geometric properties -- in particular, the affine dimension of a submanifold of Euclidean space is always less than it's dimension.

When $M$ is a manifold, all efficient covers of $M$ by ellipsoids are covers by slabs $\{ \Theta_i \}$ of the form covered in Lemma \ref{Lemmaoaijdwoiwajdioaw}. Now suppose that in our proof of the Fourier dimension bounds, there is some quantity $a > 0$ so that a generic slab $\Theta_i$ in the covers we consider has $|\Theta_i| \approx R_i^{-a}$. This is the case in both the proofs we give in this paper utilizing Lemma \ref{Lemmaoaijdwoiwajdioaw}; in Proposition \ref{lemma:UpperBoundFourierDimensionProposition} we have $a = m/2 + k - 1/6$, where $d = m + k$, and in Theorem \ref{thm:zhutheorem} we have $a = (n-k)/2 + 1$, where $d = n + 1$. Under this assumption, 
\begin{equation}
    \sum\nolimits_i R_i^{-s/2} \approx \sum\nolimits_i |\Theta_i|^{s/2a}.
\end{equation}
Thus, heuristically, if our method gives a Fourier dimension bound $\fordim(M) \leq s$, then $M$ should also satisfy an affine dimension bound $\adim(M) \leq ds/2a$. 
%
% d(3m - 1)/(3m - 1 + 6k)
% (n+1)(n-k)/((n-k) + 2)
%
If one is to investigate the limitations of the method of this paper as to obtaining sharp Fourier dimension bounds for manifolds, a natural question to investigate is as to whether there exists a manifold $M$ such that whenever we consider a cover $\{ \Theta_i \}$ of $M$ such that a generic slab has $|\Theta_i| \approx R_i^{-a}$, then $\adim(M) < (d/2a) \fordim(M)$. We are unaware if such a manifold exists at this time.

\section{Manifolds of High Codimension Cannot Be Salem} \label{sec:aowdjwaoidjioawdjmanifolds_of_high_codimension_cannot_be_salem}

Using the bounds obtained in Lemma \ref{Lemmaoaijdwoiwajdioaw}, we now prove the only if condition in Theorem \ref{thm:blehtheorem}, which as a result also finishes the proof of Theorem \ref{thm:mainTheorem}.

\begin{prop} \label{lemma:UpperBoundFourierDimensionProposition}
    Let $M$ be an $m$-dimensional $C^{2 + \alpha}$-submanifold of $\RR^{m+k}$, where $0 < \alpha \leq 1$. Suppose that for each $z_0 \in M$, there exists $\zeta_0 \in (T_{z_0} M)^\perp$ and $v_0 \in T_{z_0}M$ so that the first and second directional derivatives of the map $z \mapsto \zeta_0 \cdot (z - z_0)$ vanish in the direction of $v_0$ at $z_0$. Then $M$ has Fourier dimension at most $m - \alpha/(\alpha + 2)$, and thus is not a Salem set.
\end{prop}
\begin{proof}
    Let $\mu$ be a finite Borel measure supported on $M$, and suppose there exists $C > 0$ and $s > m - \alpha / (\alpha + 2)$ so that $|\widehat{\mu}(\zeta)| \leq C |\zeta|^{-s/2}$ for all $\xi \neq 0$. Our goal is to show that this can only hold if $\mu = 0$.

    If $\chi \in C_c^\infty(\RR^d)$, then
    \begin{equation} \label{aoidjawoidjioqe12e321e12e12d23fvwecew}
        |\widehat{\mu \chi}(\zeta)| = |(\widehat{\mu} * \widehat{\chi})(\zeta)| \lesssim |\zeta|^{-s/2}.
    \end{equation}
    By replacing $\mu$ with $\mu \chi$ for some $\chi \geq 0$ if necessary, using \eqref{aoidjawoidjioqe12e321e12e12d23fvwecew} we may assume without loss of generality that $\mu$ is compactly supported on an arbitrarily small subset of $M$. In particular, we may assume that $\mu$ is supported on a closed ball $B_0 \subset \RR^{m+k}$ small enough that for each $z_0 \in M$, the map
    \begin{equation}
        \Pi_{z_0}: M \to T_{z_0} M
    \end{equation}
    given by orthogonal projection onto $z_0 + T_{z_0} M$ is a diffeomorphism when restricted to $M \cap B_0^*$, where $B_0^*$ is a ball with the same center as $B_0$ but twice the radius. By rescaling, we may assume without loss of generality that $B_0^*$ is a unit ball. Later, we will also use the projection
    \begin{equation}
        \Pi_{z_0}^\perp: M \to (T_{z_0} M)^\perp
    \end{equation}
    of $M$ onto $z_0 + (T_{z_0} M)^\perp$. We let
    \begin{equation}
        L_{z_0}: B_{z_0} \to M
    \end{equation}
    denote the inverse of $\Pi_{z_0}$, restricted to the unit ball $B_{z_0}$ in $T_{z_0} M$. Then $L_{z_0}(B_{z_0})$ contains $M \cap B_0$, and since $M$ is $C^{2 + \alpha}$, we have the H\"{o}lder-space norm bound $\| L_{z_0} \|_{C^{2,\alpha}(B_{z_0})} \lesssim 1$, where the implicit constant is uniform for $z_0 \in M \cap B_0$.

    Let $\beta = (\alpha + 2)^{-1}$. Then we claim that for all radius $R^{-\beta}$ balls $B \subset \RR^{m+k}$ with center in $M$,
    \begin{equation} \label{eq:aiwodjoaiwjdoiawjio213j41oi2j3io12j}
        \mu(B) \lesssim R^{(m-1)(1/2 - \beta) - s/2}.
    \end{equation}
    Since the support of $\mu$ is covered by $O(R^{m \beta})$ such balls, it follows from \eqref{eq:aiwodjoaiwjdoiawjio213j41oi2j3io12j} that
    \begin{equation} \label{eq:aiojdoaiwjdoiawjediowaqjeoijq}
        \mu(M) \lesssim R^{m \beta + (m-1)(1/2 - \beta) - s/2}.
    \end{equation}
    If $s > m - \alpha/(\alpha + 2)$, then $m \beta + (m-1)(1/2 - \beta) - s/2 < 0$, so taking $R \to \infty$ in \eqref{eq:aiojdoaiwjdoiawjediowaqjeoijq} yields $\mu(M) = 0$, as was required. We will obtain this bounds on $B$ by further decomposing $B$ into \emph{grains}, i.e. thin tubes to which we can apply Lemma \ref{Lemmaoaijdwoiwajdioaw} and which are parallel, thus efficiently covering $B$ (see Figure \ref{figureballcover}).

    \begin{figure}
\begin{tikzpicture}[scale=1]

% A helper macro to draw one ball with tubes
%  #1 = center x
%  #2 = center y
%  #3 = tube angle
%  #4 = radius
\newcommand{\ball}[4]{
    \begin{scope}
        \clip (#1,#2) circle (#4);
        % draw parallel lines
        \foreach \i in {-30,...,30} {
            \draw[rotate around={#3:(#1,#2)},line width=0.3mm] 
                (#1-2,#2+\i*0.15) -- (#1+2,#2+\i*0.15);
        }
    \end{scope}
    \draw[line width=0.5mm] (#1,#2) circle (#4);
}

% Now place several balls with different tube orientations
\ball{0}{0}{0}{0.55}
\ball{1.1}{0}{70}{0.55}
\ball{-1.1}{0}{40}{0.55}
\ball{0.55}{0.95}{20}{0.55}
\ball{-0.55}{0.95}{-20}{0.55}
\ball{0.55}{-0.95}{-40}{0.55}
\ball{-0.55}{-0.95}{60}{0.55}

\end{tikzpicture}
\caption{We cover our manifold by balls of radius $R^{-\beta}$. To bound the measure of each ball, we further decompose the balls into \emph{grains}, families of parallel $R^{-1/2} \times R^{-\beta}$ tubes.} \label{figureballcover}
\end{figure}
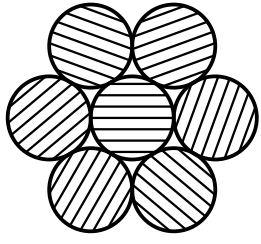

    Let $B \subset \RR^{m+k}$ be a radius $R^{-\beta}$ ball, and let $z_B \in M$ denote it's center. By assumption, we can find unit vectors $\zeta_B \in (T_{z_B} M)^\perp$ and $v_B \in T_{z_B} M$ so that the first and second directional derivatives of the function $z \mapsto \zeta_B \cdot (z - z_B)$ vanish in the direction of $v_B$ at $z_B$. For each $z_0 \in B$, find a unit vector $\zeta_0 \in (T_{z_0} M)^\perp$ such that $\Pi_{z_B}^\perp \zeta_0$ is parallel to $\zeta_B$, and find a unit vector $v_0 \in T_{z_0} M$ such that $\Pi_{z_B} v_0$ is parallel to $v_B$. Then define $\Delta: B_{z_0} \to \RR$ by setting $\Delta(x) = \zeta_B \cdot ( L_{z_0}(x) - z_0 )$. To simplify notation, let $H(x)$ denote the Hessian of $f$ at $x$. By construction, $D\Delta(0) = 0$ and if $L_{z_0}(x_B) = z_B$, then $H(x_B) v_0 = 0$.

    Using the $C^2$ control on $L_{z_0}$, we find that for any $t \in \RR$ and $w \in T_{z_0} M$,
    \begin{equation} \label{eq:aiodjoaiwjdioqwdjwioqjdiowq}
        |(tv_0 + w) H(x_B) (tv_0 + w)| = |w H(x_B) w| \lesssim |w|^2.
    \end{equation}
    Since $|z - z_0| \lesssim R^{-\beta}$, it follows that $|x_B| \lesssim R^{-\beta}$. So by the $C^{2,\alpha}$ control on $L_{z_0}$, we conclude that if $|x| \leq 2R^{-\beta}$, then for any vector $w \in T_{z_0} M$,
    \begin{equation} \label{eq:ioajcoiacnioeqwjcnioewqhnfqeiwo}
        \big| w^t (H(x) - H(x_B)) w \big| \lesssim R^{-\alpha \beta} |w|^2.
    \end{equation}
    Combining \eqref{eq:aiodjoaiwjdioqwdjwioqjdiowq} and \eqref{eq:ioajcoiacnioeqwjcnioewqhnfqeiwo}, we conclude that if $c$ is sufficiently small, and we consider the ellipsoid
    \begin{equation}
        Q = \{ tv_0 + sw : |w| = 1\ \text{and}\ v_0 \cdot w = 0\ \text{and}\ (R^{\beta} t)^2 + (R^{1/2} s)^2 \leq c \}.
    \end{equation}
    then for $x,v \in Q^*$, the ellipsoid obtained by dilating $Q$ by a factor of two, $|v^t H(x) v| \leq 1/R$. And so if $x \in Q^*$ then by Taylor's theorem
    \begin{equation} \label{eq:adawioudjioaqwjdioawjdoiawjoi123123213}
        |\Delta(x)| \leq \int_0^1 |x^T H(tx) x|\; dt \leq 1/R.
    \end{equation}
    This is sufficient control on $\Delta$ to verify the assumptions of Lemma \ref{Lemmaoaijdwoiwajdioaw}.

    Let $S = M \cap B_0$. Then if $z \in S \cap \Pi^{-1}(Q^*)$, then $z = f(x)$ for some $x \in Q^*$, and so \eqref{eq:adawioudjioaqwjdioawjdoiawjoi123123213} implies $|\zeta_0 \cdot (z - z_0)| = |\Delta(x)| \leq 1/R$. Thus Lemma \ref{Lemmaoaijdwoiwajdioaw} applies, with $\varepsilon$ some fixed  positive value, and we conclude that $\mu(\Theta) \lesssim R^{-s/2}$, where $\Theta$ is the tube defined in Lemma \ref{Lemmaoaijdwoiwajdioaw}. For each $z_0$, the tube $\Theta$ is centered at $z_0$ and pointing in the direction $v_0$, where $\Pi_{z_B} v_0$ is parallel to $v_B$. This means that the tubes have the property that either $\Theta$ and $\Theta'$ are disjoint, or $M \cap \Theta'$ is contained in $M \cap \Theta^*$. But this means we can find a family of tubes $\{ \Theta \}$ of the form above covering $B$ such that $\{ M \cap \Theta \}$ has the finite intersection property. Since $H^m(M \cap \Theta) \sim R^{-(m-1)/2 - \beta}$ and $H^m(B) \sim R^{- \beta m}$, a volumetric argument tells us that it takes $O(R^{(m-1)(1/2 - \beta)})$ such tubes to cover $B$, and so we conclude that $\mu(B) \lesssim R^{(m-1)(1/2 - \beta) - s/2}$, as was required, which completes the proof.
\end{proof}

\begin{remark} \label{remark:aowdjwaoidjaiowjd} \normalfont
    If $M$ is a $C^2$ manifold, but not $C^{2 + \alpha}$ for any $\alpha > 0$, then a modification of the method of Proposition \ref{lemma:UpperBoundFourierDimensionProposition} %then one can obtain a modification of this proof, using uniform continuity of the second derivatives of a parameterization of $M$ rather than H\"{o}lder continuity,
    can be used to argue that $M$ does not support a probability measure $\mu$ such that $|\widehat{\mu}(\zeta)| \lesssim |\zeta|^{-m/2}$ for all $\zeta \in \RR^{m+k}$. But this is insufficient to show $M$ is a Salem set since it does not contradict that $|\widehat{\mu}(\zeta)| \lesssim |\zeta|^{\varepsilon - m/2}$ for all $\varepsilon > 0$.
\end{remark}

Since Lemma \ref{lemma:StationaryDegenerateHighCodimensionLemma} implies any $m$-dimensional $C^{2,\alpha}$-manifold in $\RR^{m+k}$ for $k > \rho(m/2) + 1$ satisfies the assumptions of Proposition \ref{lemma:UpperBoundFourierDimensionProposition}, this proposition, combined with Proposition \ref{prop:ExistenceofSalemCodimension}, completes the proof of Theorem \ref{thm:mainTheorem}.

\section{The Fourier Dimensions of Nondegenerate Curves}

We now use the methods used in the proof of Theorem \ref{thm:mainTheorem} to obtain an alternate proof of a result of Junjie Zhu \cite{ZhuRank} on the Fourier dimension of hypersurfaces with a fixed number of principal curvatures, using the fact that such hypersurfaces are efficiently covered by $k$-dimensional hyperplanes. As mentioned in the introduction, Zhu's proof uses a `affine smoothing argument'; given a measure with a given Fourier decay, the method utilizes the approximate invariance of the manifold under affine symmetries to obtain a smooth function with the same Fourier decay, to which we can employ the theory of oscillatory integrals.

To demonstrate the affine smoothing strategy, and discuss it's limitations, we use the method to determine the Fourier dimension of the moment curve in $\RR^3$, i.e. the curve parameterized by the function $\gamma(t) = (t,t^2,t^3)$, a model example of a nondegenerate curve. The method of proof given here emerged from discussions with Kyle Hambrook, Chun-Kit Lai, and Jaume de Dios Pont at the 2023 Harmonic Analysis and Fractal Sets conference at The Ohio State University.

\begin{prop}
    The moment curve in $\RR^3$ has Fourier dimension at most $2/3$.
\end{prop}
\begin{proof}
    Consider the parameterization $\gamma(t) = (t,t^2,t^3)$. It suffices to show that for all $s > 2/3$, no measure $\mu$ on the real line satisfies a bound of the form
    \begin{equation} \left| \int_{-\infty}^\infty e^{2 \pi i \zeta \cdot \phi(t)} d\mu(t) \right| \lesssim |\zeta|^{-s/2} \quad\text{for all $\zeta \in \RR^3$}. \end{equation}
    Suppose such a measure existed. We note that $\gamma(t + s) = \gamma(s) + M(s) \gamma(t)$, where
    \begin{equation}
        M(s) = \begin{pmatrix} 1 & 0 & 0 \\ 2s & 1 & 0 \\ 3s^2 & 3s & 1 \end{pmatrix}.
    \end{equation}
    It follows that if we fix a non-negative, smooth, compactly supported function $\chi$, and define the smooth function $f = \chi * \mu$, then
    \begin{equation} \label{aiodjwaoidjwqioje12oi312312}
    \begin{split}
        \left|\int e^{2 \pi i \zeta \cdot \gamma(t)} f(t)\; dt \right| &= \left| \int e^{2 \pi i \zeta \cdot \gamma(t+s)} \chi(s) d\mu(t)\; ds \right|\\
        &\leq \int \chi(s) \left| \int e^{2 \pi i \zeta \cdot \gamma(t + s)} d\mu(t) \right|\; ds\\
        &\leq \int \chi(s) \left| \int e^{2 \pi i M(s)^t \zeta \cdot \gamma(t)} d\mu(t) \right|\; ds\\
        &\lesssim \int \chi(s) |M(s)^t \zeta|^{-s/2}\\
        &\lesssim |\zeta|^{-s/2}.
    \end{split}
    \end{equation}
    Now that we have a smooth meaure with the appropriate Fourier decay, we may apply a localization argument to assume this function is supported on an arbitrarily small subset of space. Once this localization is done, and $t_0$ is chosen so that $f(t_0) > 0$, then the method of stationary phase\footnote{See, e.g. Chapter VIII, Proposition 3 of \cite{Stein}.} implies that for all vectors $\zeta$ with $\zeta \cdot \gamma'(t_0) = \zeta \cdot \gamma''(t_0) = 0$,
    \begin{equation} \label{aiowjdwaoidjaqwoije12321321}
        \int e^{2 \pi i \lambda \zeta \cdot \gamma(t)} f(t)\; dt \sim |\zeta|^{-1/3}.
    \end{equation}
    Since \eqref{aiowjdwaoidjaqwoije12321321} contradicts \eqref{aiodjwaoidjwqioje12oi312312} if $s > 2/3$ and $|\zeta|$ is taken appropriately large, this completes the proof.
\end{proof}

The proof clearly generalizes to the moment curve $\gamma(t) = (t,,\dots,t^n)$ in $\RR^n$. However, it is difficult to see how it could apply to other nondegenerate curves, since these do not have the property that they can be parameterized by a function $\gamma$ such that for each $s$, the curve $t \mapsto \gamma(t + s)$ is an affine transformation of the curve $t \mapsto \gamma(t)$. The methods considered thus apply in certain situations where the affine smoothing strategy seems impractical.

\begin{prop} \label{prop:oiwajdoiwajdoiawjdoiawlabel}
    If $n \geq 2$, and if $C \subset \RR^n$ is a $C^n$-curve, then $\fordim(C) \leq 2/n$.
\end{prop}
\begin{proof}
    Let $C$ be a curve in $\RR^n$, and let $\mu$ be a measure supported on $C$ such that $|\widehat{\mu}(\zeta)| \lesssim |\zeta|^{-s/2}$, where $s > 2/n$. Our goal is to show that $\mu = 0$. By localization, we may assume $\mu$ is supported on a small ball $B_0$, small enough that for the maps $L_{z_0}: B_{z_0} \to C$ as defined in the proof of Proposition \ref{lemma:UpperBoundFourierDimensionProposition} are immersions, and $C \cap B_0 \subset L(B_{z_0})$ for each $z_0 \in C \cap B_0$. For each $z_0 \in C \cap B_0$, pick $\zeta_0 \in \RR^d$ so that the function $z \mapsto \zeta_0 \cdot (z - z_0)$ vanishes to order $n$ at $z_0$. If we define $\Delta: B_{z_0} \to \RR$ by setting $\Delta(x) = \zeta_0 \cdot (L_{z_0}(x) - z_0)$, then $\partial_t^j \Delta(0) = 0$ for $0 \leq j < n$, and so it follows that if $c > 0$ is suitably small, and we consider the interval $I \subset T_{z_0} C$ of length $c R^{-1/n}$, then $|\Delta(x)| \leq 1/R$ for all $x \in I$. Thus Lemma \ref{Lemmaoaijdwoiwajdioaw} applies, with $\varepsilon$ to an arbitrary positive value, and we conclude that $\mu(\Theta) \lesssim R^{-s/2}$, where $\Theta$ is the tube defined in Lemma \ref{Lemmaoaijdwoiwajdioaw}. The section of the curve upon which $\mu$ is supported is covered by $O(R^{1/n})$ such tubes, and so it follows that $\mu(C) \lesssim R^{1/n - s/2}$. If $s > 2/n$, then it follows that $\mu = 0$, as was required.
\end{proof}

Combining the lower bound of Proposition \ref{prop:oiwajdoiwajdoiawjdoiawlabel} with the upper bound of Proposition \ref{prop:dkpoawkdpoawdkpaowkkvmdsovselabel} completes the proof of Theorem \ref{momentcurvetheorem}.

\section{The Fourier Dimension of Hypersurfaces of Constant Nullity}

Finally, we obtain the new proof of upper bounds for the Fourier dimension of hypersurfaces of constant nullity.

\begin{prop} \label{prop:oaiwjdioawjdioawojidaoiwjdiaowdjioalabel}
    Let $M \subset \RR^{n+1}$ be a $C^2$ hypersurface, such that at each point of $M$, exactly $k$ principal curvatures vanish. Then $M$ has Fourier dimension at most $n - k$.
\end{prop}
\begin{proof}
    We employ the geometric aspects of the hypersurface $M$. It is a well known fact in the differential geometry literature that locally, a manifold $M$ with the properties above can be foliated, or ruled, by $k$-dimensional planes, and the normal vector to $M$ is constant along each foliation. Such a manifold is called a \emph{manifold of constant nullity} in the literature. A proof of this fact can be found in Lemma 3.1 of \cite{Hartman}.

    Now suppose $\mu$ is a finite Borel measure supported on $M$ with $|\widehat{\mu}(\zeta)| \leq C |\zeta|^{-s/2}$, with $s > n-k$. As in Proposition \ref{lemma:UpperBoundFourierDimensionProposition}, we may localize $\mu$, and in what follows we assume the support of $\mu$ is compactly contained within a subset of $M$ small enough that the entirety of the support of $\mu$ can be foliated by $k$-dimensional planes. We may also assume that $\mu$ is supported on a ball $B_0 \subset \RR^{m+k}$ with small enough radius that the maps $L_{z_0}: B_{z_0} \to M$ as defined in the proof of Proposition \ref{lemma:UpperBoundFourierDimensionProposition} are immersions which satisfy $M \cap B_0 \subset L_{z_0}(B_{z_0})$ for each $z_0 \in M \cap B_0$.

\begin{figure}
\begin{tikzpicture}[scale=1.8, line cap=round]
    % Parameters
    \def\n{30}          % number of tangent lines
    \def\R{0.5}         % radius of helix projection
    \def\H{15}        % vertical rise per step
    \def\L{1.2}         % tangent length

    % Draw tangent lines
    \foreach \i in {0,...,\n} {
        \pgfmathsetmacro{\t}{\i*\H}            % angle in degrees
        \pgfmathsetmacro{\x}{\R*cos(\t)}       % point on helix
        \pgfmathsetmacro{\y}{\R*sin(\t)}
        \pgfmathsetmacro{\dx}{-sin(\t)}        % tangent direction in xy
        \pgfmathsetmacro{\dy}{cos(\t)}

        % draw the tangent line segment
        \draw[line width = 0.3mm]
            (\x,\y) --
            ({\x + \L*\dx},{\y + \L*\dy});
    }

    % Draw the projected curve (optional)
    \draw[line width=0.5mm] plot[domain=0:360,samples=100]
        ({\R*cos(\x)},{\R*sin(\x)});
\end{tikzpicture}
\caption{This diagram is a projection of the \emph{tangent developable} of the helix $t \mapsto (\cos(t), \sin(t), t)$, i.e. the ruled surface given by the envelope of tangent lines to the helix. We use Lemma \ref{Lemmaoaijdwoiwajdioaw} to bound the measure of the strips between any two tangent lines coming from a $R^{-1/2}$ separated family of points on the curve, which efficiently bounds the measure of the entire surface.}
\end{figure}
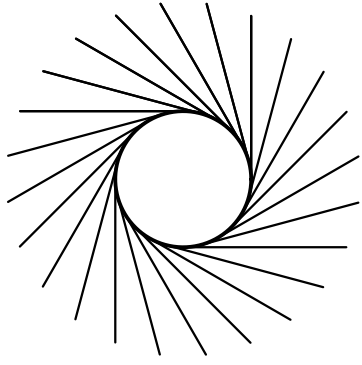

    Given $z_0 \in B_0$, write $T_{z_0} M = U_{z_0} \oplus U_{z_0}^\perp$, where $U_{z_0}$ is the $k$-dimensional space of tangents to the foliation of $M$, and let $\zeta_0$ be the normal vector to $M$ along $U_{z_0}$. Then the function $z \mapsto \zeta_0 \cdot (z - z_0)$ vanishes to order two at all points along the $k$-dimensional plane $M \cap (z_0 + U_{z_0})$. Define $\Delta: B_{z_0} \to \RR$ by setting $\Delta(x) = \zeta_0 \cdot (L_{z_0}(x) - z_0)$. Then $\Delta(x) = 0$ and $D \Delta(x) = 0$ for $x \in U_{z_0} \cap B_{z_0}$. So it follows that if $c > 0$ is suitably small, and we define the ellipsoid $Q \subset T_{z_0} M$ by
    \begin{equation}
        Q = \{ t v + s w : v \in U_{z_0}\ \text{and}\ w \in U_{z_0}^\perp\ \text{and}\ t^2 + (R^{1/2} w)^2 \leq c \},
    \end{equation}
    then $|\Delta(x)| \leq 1/R$ for all $x \in Q^*$. Thus Lemma \ref{Lemmaoaijdwoiwajdioaw} applies, with $\varepsilon$ to an arbitrary positive value, and we conclude that $\mu(\Theta) \lesssim R^{-s/2}$, where $\Theta$ is the slab defined in Lemma \ref{Lemmaoaijdwoiwajdioaw}. Now $\Theta$ contains a strip with length $\Omega(1)$ along $U_{z_0}$ and length $O(R^{-1/2})$ in orthogonal directions. A volumetric argument shows that $B_0$ can be covered by $O(R^{(n-k)/2})$ such strips, and so it follows that
    \begin{equation} \label{eqoijaiodjawiodjawiodjawodjq324213421}
        \mu(M) \lesssim R^{\frac{n-k-s}{2}}.
    \end{equation}
    Since $s > n - k$, taking $R \to \infty$ in \eqref{eqoijaiodjawiodjawiodjawodjq324213421} yields that $\mu(M) = 0$, as was required. \qedhere

\end{proof}

That the Fourier dimension of hypersurfaces of constant nullity $k$ is at least $n - k$ is classical, and obtained by much the same reasoning as was performed in \eqref{sec:lower_bounds_for_the_codimension_threshold}, i.e. taking a smooth volume density on $M$, and using the theory of oscillatory integrals to analyze it's Fourier transform. The result was first obtained by Walter Littman \cite{MR155146}, who used it to analyze the Sobolev mapping properties of convolution with such measures, following analysis of Carl Herz \cite{Herz} and Edmund Hlawka \cite{Hlawka1,Hlawka2} for hypersurfaces with positive Gaussian curvature. Combining this lower bound with the upper bound of Proposition \ref{prop:oaiwjdioawjdioawojidaoiwjdiaowdjioalabel} completes the proof of Theorem \ref{thm:zhutheorem}.

\newpage

\bibliographystyle{amsplain}
\bibliography{CurvesCantBeSalem}

\end{document}